\documentclass[12pt]{article}
\usepackage{amsmath}
\usepackage{graphicx,psfrag,epsf}
\usepackage{enumerate}
\usepackage{natbib}
\usepackage{amsmath,amsthm,amsfonts,epsfig,amssymb,natbib,eucal,eufrak}
\usepackage{color}

\newtheorem{lem}{Lemma}

\newtheorem{result}{Result}
\newtheorem{comment}{Comment}
\newtheorem{remark}{Remark}

\newtheorem{pro}{Proposition}
\newcommand{\blind}{0}

\begin{document}

\def\spacingset#1{\renewcommand{\baselinestretch}%
{#1}\small\normalsize} \spacingset{1}


\if0\blind
{
  \title{\bf A note on the minimax solution for the two-stage group testing problem}
  \author{Yaakov Malinovsky \thanks{
    The work was partially supported by a 2013 UMBC Summer Faculty Fellowship grant }\hspace{.2cm}
    \thanks{Corresponding author}\\
    Department of Mathematics and Statistics\\ University of Maryland, Baltimore County, Baltimore, MD 21250, USA\\
    and \\
    Paul S. Albert\thanks{
    The work was supported by the {\it Eunice Kennedy Shriver} National Institute of Child Health and Human Development intramural program.}\hspace{.2cm}\\
    Biostatistics and Bioinformatics Branch\\ Division of Intramural Population Health Research\\ {\it Eunice Kennedy Shriver} National Institute of Child Health\\ and Human Development, Bethesda, MD 20892, USA}
  \maketitle
} \fi

\if1\blind
{
  \bigskip
  \bigskip
  \bigskip
  \begin{center}
    {\LARGE\bf A note on the minimax solution for the two-stage group testing problem}
\end{center}
  \medskip
} \fi

\bigskip
\begin{abstract}
Group testing is an active area of current research and has important applications in medicine, biotechnology, genetics, and product testing. There have been recent advances in design and estimation, but the simple Dorfman procedure introduced by R. Dorfman in 1943 is widely used in practice.
In many practical situations the exact value of the probability $p$ of being affected is unknown.
We present both minimax and Bayesian solutions for the group size problem when $p$ is unknown. For unbounded $p$ we show that the minimax solution for group size is $8$, while using a Bayesian strategy with Jeffreys prior results in a group size of $13$. We also present solutions when $p$ is bounded from above.
For the practitioner we propose strong justification for using a group size of between eight to thirteen when a constraint on $p$ is not incorporated and provide useable code for computing the minimax group size under a constrained $p$.

\end{abstract}

\noindent%
\vfill

\newpage
\spacingset{1.45} 
\section{Introduction}
\label{sec:intro}
The purpose of this article is to propose a practical and simple group testing procedure that performs well in a wide range of situations.
Group testing procedures save cost and time and have wide spread applications, including
blood screening \citep{Dorfman1943, F1964, L1994, G1994, DH2012, M2012, T2013}, quality control in product testing \citep{SG1959, SG1966}, computation biology \citep{D2005}, DNA screening \citep{Dh2006, G2012}, and photon detection \citep{v2013}. According to \cite{H2006}, group testing began as early as 1915, when it was used in dilution studies for estimating the density of organisms in a biological medium.

In his 1950 book, William Feller nicely described the group testing problem as: `` {\it A large number, $N$, of people are subject to a blood test. This can be administered in two ways. (i) Each person tested separately. In this case $N$ tests are required. (ii) The blood samples of $k$ people can be pooled and analyzed together. If the test is negative, this one test suffices for the $k$ people. If the test is positive, each of the $k$ persons must be tested separately, and in all $k+1$ tests are required for the $k$ people. Assume the probability $p$ that the test is positive is the same for all  and that people are stochastically independent.}"
Procedure $(ii)$ is commonly refered to as the Dorfman two-stage group testing procedure (DTSP) \citep{Dorfman1943, S1978}.
Interesting historical comments related to this problem can be found in the introduction of the book by \cite{Dh1999}.

Let $E\left(k, p\right)$ be the expected number of tests per person using DTSP with a group size $k$ and probability of infection $p$. Then $E\left(1, p\right)=1$ and $E\left(k, p\right)=1-(1-p)^k+k^{-1}$ for $k\geq 2$.
An important issue for the DTSP is to find an optimal value of $k$, $k^{*}=k^{*}(p)$, that minimizes the expected number of tests for a given $p$.

\cite{S1978} solved this optimization problem as follows.
Let $[x]$ and $\{x\}=x-[x]$ denote the integer and fractional parts of $x$, respectively. $k^{*}$ is a non-increasing function of $p$, which is $1$ for $p>1-1/3^{1/3}\approx 0.31$, and otherwise is either $1+[p^{-1/2}]$ or $2+[p^{-1/2}]$. If $\{p^{-1/2}\}<[p^{-1/2}]/\left(2[p^{-1/2}]+\{p^{-1/2}\}\right)$, then $k^{*}=1+[p^{-1/2}]$.
If  $\{p^{-1/2}\}>[p^{-1/2}]/\left(2[p^{-1/2}]+\{p^{-1/2}\}\right)$, then  to find out which of the values $k^{'}=1+[p^{-1/2}]$ or $k^{''}=2+[p^{-1/2}]$ is optimal, one plugs them into $E\left(k, p\right)$.

The DTSP is not an optimal procedure and can be improved by introducing more than two stages (e.g., \cite{SG1959} {for known $p$\color{blue}}). However, the optimal testing algorithm  (unknown $p$) is unknown, and it is a difficult optimization problem \citep{Dh1999}.
\cite{U1960} proved that if $p>(3-5^{1/2})/2\approx 0.38$, then there does not exist an algorithm that is better than individual one-by-one testing.
\cite{SG1966} presented a Bayesian model for the multistage group testing problem based on {\color{blue} a} prior distribution of $p$.
\cite{ST1990} derived adaptive procedures for the two-stage group testing problem based on a beta prior distribution.
Although the DTSP is not optimal, it is often used in practice due to its simplicity
\citep{T1993, M2000, W2008}.

In many practical situations the exact value $p$ of the
probability of being affected (e.g., disease prevalence)
is unknown, and therefore the optimal group size cannot be calculated.
In this note, we derive both the minimax group size as well as the Bayesian solution under reasonable prior distributions for the DTSP. A comparison of the solutions under these different alternatives will aid the practitioner in the design of future applications of group testing.
We first present the loss function needed for the minimax solution.

\section{Loss Function}
\label{Sec:L}
 If we know the value of $p$, then by using the result from \cite{S1978} we achieve the minimum value for $E\left(k, p\right)$ in DTSP by:
\begin{equation}
\label{eq:D}
E\left(k^{*}(p), p\right)=\left\{\begin{array}{ccc}
                            1-(1-p)^{k^{*}(p)}+\frac{1}{k^{*}(p)} & for & 0< p \leq 1-(1/3)^{1/3} \\
                            1 &  & otherwise,
                          \end{array}
                          \right.
\end{equation}
where we can write the expected number of tests per person in a group of size $k$ as
\begin{equation}
\label{eq:E}
E\left(k, p\right)=
\left\{\begin{array}{ccc}
                            1-(1-p)^{k}+\frac{1}{k} & for & k>1 \\
                            1 & for & k=1.
                          \end{array}
                          \right.
\end{equation}
It is important to note that according to \cite{S1978}, $k^{*}(p)$
is a non-increasing function of $p$, and $k^{*}(p)\rightarrow \infty$ as $p\downarrow 0$.
Combining this fact with \eqref{eq:D}, and using the particular form of $k^{*}(p)$ it is possible to show (see  Appendix \ref{AA:0}) that
$E\left(k^{*}(p), p\right)\rightarrow 0$, as $p\downarrow 0$.
Also, from Comment \ref{com:2} (Appendix \ref{B:b} ) it follows that $E\left(k^{*}(p), p\right)$ is a non-decreasing function of $p$.

We define the loss in DTSP as a difference between the expected number of tests and the expected number of tests under an optimal DTSP. Specifically,
\begin{equation}
\label{eq:Loss}
L\left(k, p\right)=E\left(k, p\right)-E\left(k^{*}(p), p\right).
\end{equation}

It is important to note that this type  of loss function was considered by \cite{R1952}
for  the  two-armed bandit problem. Although \cite{R1952} discussed a different problem from ours, there is a similarity between the two problems in that both try to quantify the loss ``due to
ignorance of the true state of affairs'' and to seek a minimax solution.

Another loss function is
$\displaystyle L_{2}\left(k, p\right)=\frac{E\left(k, p\right)}{E\left(k^{*}(p), p\right)}-1$,
which reflects a relative rather than absolute change. However, it cannot be used
since $\lim_{p\downarrow 0}E\left(k^{*}(p), p\right)=0$ and the measure becomes undefined when $p$ is near $0$.
Also, it is important to note that the expected number of tests per person, $E\left(k, p\right)$, itself is not an appropriate loss function for obtaining the minimax solution since $E\left(k, p\right)$ is a nondecreasing function of $p$ for any $k$ (i.e, $E\left(k, p\right)$ is maximized at the upper bound of $p$ for all $k$).




\section{Minimax Solution for Unbounded $p$ }
\label{sec:M}
We define the minimax group size as a group size that minimizes the largest loss $L\left(k, p\right)$,
\begin{equation}
\label{eq:mm}
k^{**}= \arg\min_{k \in \mathbb{N}^{+}}\sup_{p\in (0, 1-(1/3)^{1/3}]} L\left(k, p\right).
\end{equation}
From equation \eqref{eq:D} it follows that the DTSP has utility only when $p$ is in the range $(0,1-(1/3)^{1/3}]$, while individual testing is optimal outside this range.
Also it is important to note that $\displaystyle \sup_{p\in (0, 1-(1/3)^{1/3}]} L\left(k, p\right)=\sup_{p\in (0, 1)} L\left(k, p\right)$ (see Appendix \ref{AA:0}, Result \ref{res:3}).
In Appendix \ref{AAA:p}, Lemma \ref{lem:m} we are able to obtain a closed-form expression for $k^{**}$ in \eqref{eq:mm}.
The proof uses the following two steps to show that $k^{**}=8$.
\\
\noindent
{\bf Step 1}: Fix $k$, $k \in \left\{1,2,3,\ldots,\right\}$. Find  $p^{*}(k)=\arg\sup_{p \in (0, 1-(1/3)^{1/3}]}L\left(k, p\right)$.\\
\noindent
{\bf Step 2}: Find $k^{**}$, where
\begin{equation}
\label{eq:3}
k^{**}=\arg\min_{k \in \left\{1,2,3,\ldots\right\}}L\left(k, p^{*}(k)\right).
\end{equation}
We present a heuristic argument for finding the minimax solution which we present as follows (we also use the same argument when additional information on an upper bound of $p$ is incorporated: see Section \ref{sec:R}).
\noindent
In Step 1 we performed a grid search with incremental steps of $10^{-6}$.  We present the graphs of $L\left(k, p\right)$ as a function of $p\in (0, 1-(1/3)^{1/3}] $
for $k=1,\ldots,9$ (Figure \ref{fig:comp0}). We noticed that $p^{*}(1)=p^{*}(2)=\ldots=p^{*}(7)=0$ and therefore
$L\left(k, p^{*}(k)\right)=1/k$ for $k=1,2,\ldots,7$. It is clear that there is a jump at $k=8$ for $p^{*}(k)$, reflecting the worst case for $p$ (largest loss) in Step 1.
 \begin{figure}[!h]
 \centering
\centerline{\includegraphics[ width=\textwidth]{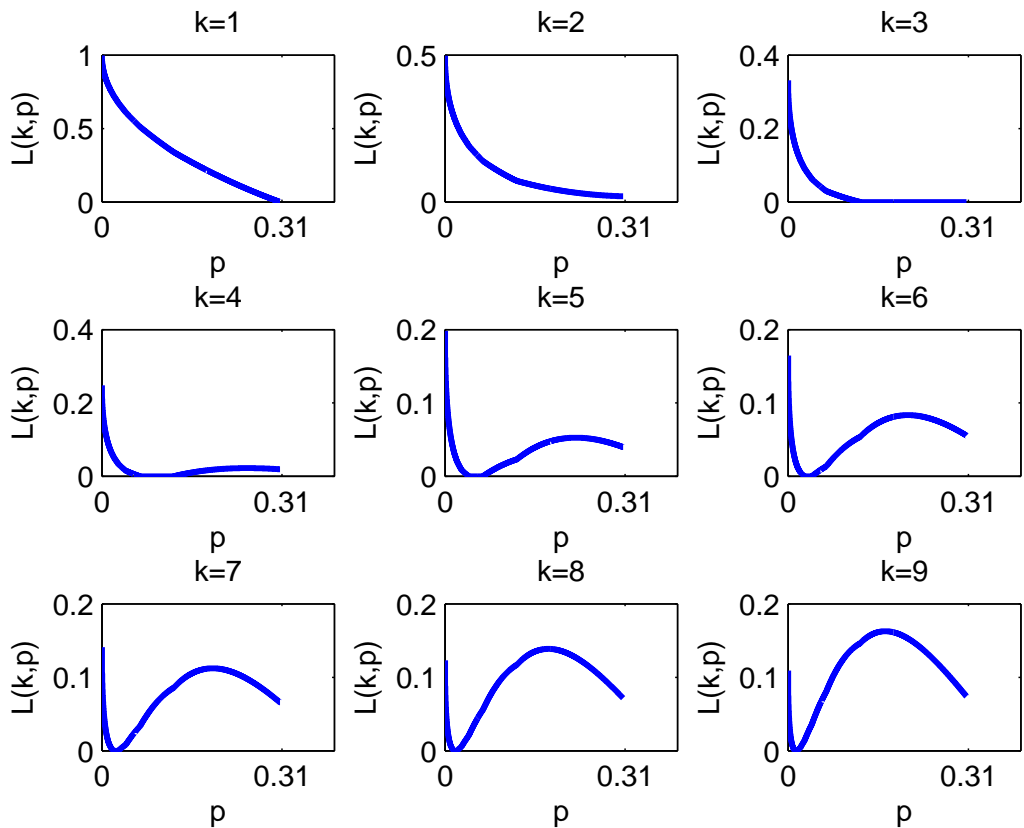}}
\caption{\small $L\left(k, p\right)$ as a function of $p\in (0, 1-(1/3)^{1/3}] $.}
\label{fig:comp0}
\end{figure}
Note that the loss function \eqref{eq:Loss} is not a smooth function of $p$. In Result \ref{res:2} (Appendix \ref{AAA:p}), we prove that $L\left(k,\,p^{*}(k) \right)$ is a unimodal function of $k$ for $k\geq 1$. We also obtain a close-form expression for $p^{*}(k)$ (see Remark \ref{Rem:1}, Appendix \ref{AAA:p}).

Table 1 presents the results of the maximization of the loss function (Step 1) as a function of $p$ for a given $k$. The minimax solution is then obtained by minimizing the loss function as a function
of $k$ (Step 2) (we present Matlab code for the two-step procedure in Appendix \ref{A:a}).

\begin{table}[!h]
\caption{Two-step Solution} 
\begin{center}
 \begin{tabular}{ c|ccccccccc}
   $k$                             & 1 & 2 & 3 &\ldots& 7 & 8 & 9 & 10  \\
   \hline
   $p^{*}(k)$                      & $0$ &$0$ & $0$ &\ldots& $0$ & $ 0.178$ & $0.167$ &  $0.158$\\
   \hline
   $L\left(k, p^{*}(k)\right)$     & $1$ & $ 1/2$ & $1/3$ & \ldots &$1/7$ & $0.138$ & $0.162$ & $0.184$  \\
\end{tabular}

\begin{tabular}{  ccccc}
    $25$ & 50 & 100 & 1000 & 10000 \\
   \hline
   $0.083$ &$ 0.049$ & $0.029$ & $ 0.004$ & $ 0.0005$ \\
   \hline
   $0.382$ & $0.516$ & $0.628$ & $0.858$ & $0.949$ \\
\end{tabular}
\label{table:1} 
\end{center}
\end{table}

As shown in Table 1, the minimax solution is $k^{**}=8$.
In order to evaluate the performance of the minimax solution, we compare it with the optimal solution assuming $p$ is known.
Specifically, we compare the optimal expected number of tests per individual $E\left(k^{*}(p), p\right)$ with the minimax expected number of tests
$E\left(k^{**}, p\right)$. Table 2 presents the ratio of these two quantities defined as $RE(p)=E\left(k^{**}, p\right)/E\left(k^{*}(p), p\right)$
for different values of $p$.
\begin{table}[!h]
\caption{Relative efficiency of minimax design} 
\begin{center}
 \begin{tabular}{ c| ccccccccc}
   $p$                             & 0.0001 & 0.0005& 0.001 & 0.005 & 0.01 & 0.05 & 0.10 & 0.25 & 0.30  \\
   \hline
   $RE(p)$                      & $6.305$ & $2.900$&$2.118$ & $1.181$ & $1.034$ & $  1.082$ & $1.169$ &  $ 1.124$ &$1.078$\\
   \hline
   $RE_{J_{1}}(p)$                      & $3.921$ & $ 1.875$&$1.432$ & $1.007$ & $ 1.020$ & $ 1.322$ & $ 1.385$ &  $ 1.156$ &$1.078$\\
\end{tabular}
\label{table:2} 
\end{center}
\end{table}
Table 2 suggests that as long as $p$ is not very small ($p<0.005$), the minimax solution is close to optimal. In the situation where we believe that $p$ is very small and we know the upper bound for $p$, we can obtain the minimax solution in the restricted parameter space. This is discussed in section \ref{sec:R}.
In Table 2 we also present the relative performance of the Bayesian solution under a Jeffrey's prior to the optimal solution as the ratio $RE_{J_{1}}(p)=E\left(k^{*}_{J_{1}}, p\right)/E\left(k^{*}(p), p\right)$, where $k^{*}_{J_{1}}$ is the optimum group size under a Jeffrey's prior. This will be discussed in section \ref{sec:U}.

In practice it is useful to identify the range of the values of $p$ for which minimax design $k^{**}=8$ is optimal.
More specifically, for any group size $l$ we can compute the range $\displaystyle [p^{\,*}_{l}, p^{\,**}_{l}]$ of the values of $p$ for which the group size $l$ is optimal.
The derivation of $p^{\,*}_{l}$ and $p^{\,**}_{l}$ is presented in Appendix \ref{B:b} as Proposition 1.
As a consequence of Proposition 1, we find that the minimax design is optimal over a range of $p$ of $\displaystyle p^{\,*}_{8}=0.0157$ to
$p^{\,**}_{8}=0.0206$.

In the next section we propose a Bayesian alternative to the minimax solution.

\section{Bayesian Solution}
\label{sec:U}
For the Baysian design we are required to specify a prior distribution $\pi(p)$ for the probability $p$ of being affected.
The basic assumption of our model is that the affected status $X_i,\,i=1,2,\ldots$ are i.i.d. Bernoulli random variables with parameter $p$ (i.e., $X_{i}\,|\,p\sim Ber(p)$).
Under the Bayesian consideration with the specified prior distribution $\pi$, the analog to the  loss function \eqref{eq:Loss} is
\begin{equation}
\label{eq:BLoss}
L_{\pi}\left(k\right)=\int L\left(k, p\right) \pi(p)\,dp=E_{\pi}\left(k\right)-E^{*}_{\pi},
\end{equation}
where $\displaystyle E_{\pi}\left(k\right)=\int {E\left(k, p\right)}\pi({p})\,dp,\,\,\,E^{*}_{\pi}=\int {E\left(k^{*}(p),p\right)}\pi({p})\,dp$.
After choosing a prior, the loss function
$L_{\pi}\left(k\right)=E_{\pi}\left(k\right)-E^{*}_{\pi}$ is a function of $k$ only, and the optimal value of $k$ is simply,
\begin{equation}
\label{eq:Bk}
\displaystyle
k^{*}_{\pi}=\min_{k}L\left(k, \pi\right)=\min_{k}E_{\pi}\left(k\right).
\end{equation}
It is clear that the optimal group size $k^{*}_{\pi}$ in  \eqref{eq:Bk} is a function of the prior distribution $\pi$.
\subsection{Uniform Prior}
If we do not have any prior information about the disease prevalence $p$, i.e., all values of $p$ are equally likely, then the uniform prior is reasonable. Denote the uniform prior by $\pi=I_U$, where $U$ is an upper bound of the distribution support, and $\displaystyle E_{\pi}\left(k\right)$
can be written as
\begin{equation}
\label{eq:I}
E_{I_{U}}\left(k\right)=1+\left\{\frac{1}{k}+\frac{1}{k+1}\left((1-U)^{k+1}-1\right)\right\}1_{\left\{k>1\right\}}.
\end{equation}
For unbounded $p$, $U=1$, $\displaystyle E_{I_{1}}\left(k\right)=1+\frac{1}{k}1_{\left\{k>1\right\}}$, and therefore  $\displaystyle k^{*}_{I_{1}}=1$.
As we will discuss in Section \ref{sec:R} (see Tables \ref{table:3}, \ref{table:4}, and \ref{table:5}), the Bayesian solution is not $1$ for bounded $p$ (i.e. $U<1$).
\subsection{Jeffreys Prior}
Another possibility is to find a prior distribution $\pi(p)$ that has {\it a small effect} on the posterior $\pi(p|x)$ distribution. Thus, we want to find a prior distribution $\pi(p)$ that produces the maximum value of the Kullback-Leibler information, $\displaystyle K\left(\pi(p\,|\,x), \pi({p})\right)$, for discrimination between two densities, $\pi(p)$ and $\displaystyle \pi(p\,|\,x)$, where the latter is the posterior that reflects the sampling density.
The Kullback-Leibler information for discrimination between two densities, $f$ and $g$ is defined as
\begin{equation}
\label{eq:KL}
K\left(f, g\right)=E_{f}\left(\log\frac{f}{g}\right).
\end{equation}
A prior that maximizes the Kullback-Leibler information is by default a {\it reference} prior
(for the general discussion of choosing a reference prior, we refer to \cite{J1946}, \cite{B2009}, and \cite{S2014}).
The above approach to the construction of the reference prior is due to \cite{B1979} (an excellent summary is given in \cite{LC1998}) and the recent developments in \citep{B2009,B2012}.
We cannot directly use $\displaystyle K\left(\pi(p\,|\,x), \pi({p})\right)$ because it is a function of $x$. Therefore, we consider the expected value of $\displaystyle K\left(\pi(p\,|\,x), \pi({p})\right)$ with respect to the marginal distribution of $X$, which is the Shannon information
\begin{equation}
S(\pi)=\int K\left(\pi(p\,|\,x), \pi({p})\right)m_{\pi}(x)dx,
\end{equation}
where $\displaystyle m_{\pi}(x)=\int f(x\,|\,p)\pi(p)dp$ is the marginal distribution of $X$.

The following lemma is due to \cite{CB1990} and is taken from \cite{LC1998}:
\begin{lem}
\label{r:C}
Let $X_1,\ldots,X_n$ be an iid sample from $f(x\,|\,p)$, and let $S_{n}(\pi)$ denote the Shannon information of the sample. Then, as $n\rightarrow\infty$,
\begin{equation}
\label{eq:SN}
S_{n}(\pi)=\frac{1}{2}\log\frac{n}{2\pi e}+\int \pi(p)\log\frac{|I_{n}(p)|^{1/2}}{\pi(p)}dp+o(1),
\end{equation}
where $I_n(p)$ is the Fisher information contained in the sample $X_1,\ldots,X_n$, and
$o(1)$ is notation of the function $h(n)=o(1)$ such that $h(n)\rightarrow 0$ as $n\rightarrow \infty$.
\end{lem}
The Fisher information contained in the sample of size $n$ is defined as
\begin{equation}
\label{eq:F}
I_{n}(p)=-E\left\{\frac{\partial^2}{\partial p^2}log f(X_1,\ldots,X_n|\,p)\right\}=\frac{n}{p(1-p)}.
\end{equation}
The last equality in \eqref{eq:F} is due to the fact that $X_i$ given $p$ follows a Bernoulli distribution.\\
\noindent
From Jensen's inequality, it follows that the right-hand side of equation \eqref{eq:SN} (approximately Shannon information due to Lemma \ref{r:C} is maximized if the prior distribution is proportional to the square root of the Fisher information, i.e.,
\begin{equation}
\label{eq:J}
\pi_{J}(p)=\arg\sup_{\pi}S_{n}(\pi)\propto |I_{n}(p)|^{1/2}\propto \frac{1}{(p(1-p))^{1/2}}.
\end{equation}

The last expression $\displaystyle \frac{1}{(p(1-p))^{1/2}}$ in \eqref{eq:J}  is known as Jeffreys prior (see for example \cite{LC1998}),
where the normalizing constant for Jeffreys prior is
\begin{equation}
\label{eq:c}
c_U=\int_{0}^{U}{\frac{1}{(p(1-p))^{1/2}}}dp=2\arcsin\sqrt{U}.
\end{equation}
Under this approach if $\displaystyle 0<p\leq U\leq 1$, then $\displaystyle \pi_{J}(p)=c_{U}\frac{1}{(p(1-p))^{1/2}}$.

Under this prior, the expected number of tests per person with a group size $k$ is
\begin{equation}
\label{eq:Ex}
E_{J_{U}}\left(k\right)=\frac{1}{c_U}\int_{0}^{U}E(k,p)\frac{1}{(p(1-p))^{1/2}}\,dp.
\end{equation}
Numerically evaluating \eqref{eq:Ex} (Appendix \ref{A:a}), the optimum group size under a Jeffreys prior is
$$\displaystyle k^{*}_{J_{1}}=\arg\min_{k}E_{{J_{1}}}\left(k\right)=13.$$

First, from a Bayesian perspective, the minimax design performs well relative to the optimal Bayesian design under a Jeffreys prior
(i.e., $\displaystyle E_{{J_{1}}}\left(13\right)/E_{{J_{1}}}\left(8\right)=0.9219/0.9286=0.993$). From a frequentist point of view, the Bayesian design performs well
for small $p$ but not as well as the minimax design for $p>0.01$
(see third row of the Table 2, where $RE_{J_{1}}(p)=E\left(k^{*}_{J_{1}}, p\right)/E\left(k^{*}(p), p\right)$).

It is important to note that theoretically we can define the minimax group size under the Bayesian setup (see for example \cite{F1967}, page 57) in the following way
$$\displaystyle k_{B}^{*}=\min_k\sup_{\pi}L_{\pi}\left(k\right).$$
Unfortunately, maximization with respect to all possible prior distributions $\pi(p)$ is intractable.
Even if we consider a beta prior, the problem remans intractable unless we limit the range of the parameters.

 \section{Minimax and Bayesian Solutions for Bounded $p$ }
\label{sec:R}
Define the minimax solution in the restricted parameter space (when we know an upper bound $U$ of $p$) as
\begin{equation}
\label{eq:mmr}
k_{U}^{**}=\arg\min_{k \in \mathbb{N}^{+}}\sup_{p\in (0, U]}L\left(k, p\right).
\end{equation}
The minimax solution subject to an upper bound on $p$ denoted as $k_{U}^{**}$ can be evaluated with a two-step procedure similar as to that presented in Section \ref{sec:M}.
The difference is that the support of $p$ is changed to $(0,U]$ and the grid step is changed accordingly.
The following table
demonstrates both  minimax and Bayesian solutions in the
restricted parameter space of $p$ (only the upper bound is specified).
\begin{table}[!h]
\caption{Minimax and Bayesian solution when upper bound $U$ of $p$ is specified} 
\centering 
\begin{center}
 \begin{tabular}{ c|ccccccccc}
   $U$                             & 0.0001 & 0.0005& 0.001 & 0.005 & 0.01 & 0.05 & 0.10 & 0.15 &0.30  \\
   \hline
   $k_{U}^{**}$                      & $201$ &   91   &      $64$ & $30$ & $21$ & $  11$ & $8$ &  $ 8$ & $8$\\
   \hline
   $k_{I_{U}}^{*}$                      &142 &    64 &  45     &21  & 15& 7 & 5 &5& 4\\
   \hline
   $k_{J_{U}}^{*}$                      &181 &  79   &56       &25  &18 & 9 & 7 & 6  &5 \\
\end{tabular}
\label{table:3} 
\end{center}
\end{table}

Define $RE_{U}(p)=E\left(k_{U}^{**}, p\right)/E\left(k^{*}(p), p\right)$ as an index of the efficiency of the minimax test (in the restricted parameter space) relative to that of the optimal DTSP test. In the Table 4 we present $RE_{U}(p)$ for different values of $U$ and $p$.
\begin{table}[!h]
\caption{Relative efficiencies of restricted minimax and Bayesian designs} 
\begin{center}
 \begin{tabular}{ c|c cc|ccc|ccccc}
 \hline
\multicolumn{1}{c}{}&\multicolumn{3}{ c| }{$U=0.0005$}&\multicolumn{3}{ c| }{$U=0.005$}&\multicolumn{3}{ c }{$U=0.05$} \\
\hline

 $p$                             & $0.0001$ & $0.0003$ & $0.0005$ & $0.001$ & $0.003$ &$0.005$  & $0.005$  &$0.01$  & $0.05$  \\
   \hline
 $RE_{U}(p)$                      & $1.0048$ &$1.0994$ & $1.2474$ & $1.0028$ & $1.1055$ & $1.2433$ &  $1.0392$ & $1$ & $1.2249$\\
 \hline
 $RE_{I_{U}}(p)$                      & $1.1030$ &$1.0044$ & $1.0596$ & $1.0901$ & $1.0060$ & $1.0606$ &  $1.2749$ & $1.0778$ & $ 1.0429$\\
 \hline
 $RE_{J_{U}}(p)$                      & $1.0289$ &$1.0461$ & $1.1556$ & $1.0310$ & $1.0392$ & $1.1343$ &  $1.1159$ & $ 1.0103$ & $1.1282$\\
   \hline
 $k^{*}(p)$                       &$101$&$58$&$45$&$32$&$19$&$15$&$15$&$11$&$5$\\
 \hline
 $k^{**}_{U}$                   &$91$&$91$&$91$&$30$&$30$&$30$&$11$&$11$&$11$\\
 \hline
 $k^{*}_{I_{U}}$                   &$64$&$64$&$64$&$21$&$21$&$21$&$7$&$7$&$7$\\
 \hline
 $k^{*}_{J_{U}}$                   &$79$&$79$&$79$&$25$&$25$&$25$&$9$&$9$&$9$\\
\end{tabular}
\label{table:4} 
\end{center}
\end{table}

A comparison between Tables 2 and 4 demonstrates the clear advantage of a minimax estimator in the restricted parameter space
in comparison to the minimax estimator in the unrestricted parameter space. For example, when $p=0.001$, bounding the parameter space by $U=0.005$ increases the relative efficiency by $2.11 (=2.118/1.0028)$. Table 4 shows that the minimax solution in the restricted parameter space performs the worst relative to an optimal design when $p=U$.

In Table 4 we also present the performance of the optimal Bayesian (under Uniform and Jeffreys priors) design in the restricted parameter space. In the restricted parameter space, the optimal group size under Jeffreys prior is
$\displaystyle k^{*}_{J_{U}}=\arg\min_{k}E_{{J_{U}}}\left(k\right),$
where $\displaystyle E_{{J_{U}}}\left(k\right)$ can be evaluated using \eqref{eq:c} and \eqref{eq:Ex} for a particular
upper bound $U$.
The ninth row in Table 4 presents $k^{*}_{J_{U}}$, and the fifth row presents the relative efficiency
$\displaystyle RE_{J_{U}}(p)=E(k^{*}_{J_{U}},p)/E(k^{*}(p),p)$.
The optimal Bayesian group size for bounded $p$ is similar to the minimax group size, with the relative efficiencies of both designs being near optimum.

\section{Summary}
This note presents both unconstrained and constrained minimax group size solutions for group testing within the DTSP framework.  We found that a group size of eight is the unconstrained minimax solution. We also present novel methodology to evaluate the range of the values of $p$ for which the minimax design is optimal.
When we have prior information that establishes an upper bound on $p$, we show that the constrained performs substantially better than the unconstrained minimax.
An advantage of the minimax solution is their simplicity within the two-stage group testing framework.
In addition, we developed a Bayesian design under both constrained and unconstrained settings which in most cases
performed similarly to the minimax design in terms of the relative efficiency.
This paper has important design implications in practice.
For example, \cite{P2004} used a DTSP with a group size of $10$ to detect acute HIV infection.
This is consistent with
our design result that suggested a group size of between $8$ and $13$ given that no information on the primary infection rate ($p$) was known a priori.
Further, with known constraint on support of $p$, we provide computer code (Appendix \ref{A:a}) that easily can be applied by the practitioner.
Research in more general algorithms (e.g., more than two-stage) is needed,
but any such algorithm will be complex and difficult for the practitioner to implement.

\section*{Acknowledgement}
An Editor, Associate Editor and two referees made thoughtful and constructive
comments and suggestions that resulted in very significant improvements in the paper.
The authors thank Sara Joslyn for editing the paper and also thank Abram Kagan and Yosef Rinott for discussions on the topic and comments on the manuscript.






\appendix
\section{Inverse of the Samuels Result }
\label{B:b}
\begin{pro}
\label{eq:prop2}
Let $q=1-p$. Define $q^{*}_{2}=1/3^{1/3}$ and $q^{*}_{l}, l\geq 3$ be the larger real root of equation $q^{l}(1-q)=\frac{1}{l(l+1)}$.
If\, $q\leq 1/3^{1/3}$, then $k^{*}(q)=1$. If $q>1/3^{1/3}$ and $q^{*}_{l-1}<q<q^{*}_{l},\, l\geq 3$, then $k^{*}(q)=l$.
\end{pro}

For example, for the minimax group size $\displaystyle k^{**}=8$, $\displaystyle q^{*}_{8}\approx1-0.0157$ is the larger real root of equation $\displaystyle q^{8}(1-q)-\frac{1}{8(8+1)}=0$, and $\displaystyle q^{*}_{7}\approx 1-0.0206$ is the larger real root of equation $\displaystyle q^{7}(1-q)-\frac{1}{7(7+1)}=0$.
Therefore,  the minimax group size $\displaystyle k^{**}=8$ is optimal for any $q$ in the range $\displaystyle \left(1-0.0206, 1-0.0157\right)$ or, alternatively,
for any $p$ in the range $\displaystyle \left(0.0157, 0.0206\right)$.


\begin{comment}
The Corollary in \cite{S1978} (equation $(2)$) provides the lower and upper bounds for  $E\left(k^{*}(p), p\right)$. The upper bound is not sharp, and in fact is invalid since it is greater than one when $p$ is larger than 0.1485 (since $0.1485$ is the real zero of $4x^3+2x-1, x=p^{1/2}$).
A simple upper bound for $E\left(k^{*}(p), p\right)$  follows directly from Proposition \ref{eq:prop2}:
if $p\leq 1-1/3^{1/3}$, then $E\left(k^{*}(p), p\right)<E\left(3, p\right)=1-q^3+1/3.$
\end{comment}
\noindent
{\bf Proof of Proposition \ref{eq:prop2}.}
\begin{proof}
We obtain the proof from the analyzing \cite{S1978} method. Assume that  $p\leq 1-1/3^{1/3}$.
First, from the \cite{S1978} Theorem (Section 1) it follows that $k^{*}(p)\geq 2$.
Second, Proposition 1 in \citep{S1978} says that $k^{*}(p)$ is never $2$. Therefore $k^{*}(p)$ should be at least $3$.
Third, following \cite{S1978} notation, define $r_{1}(k)$ and $r_{2}(k)$ as the smaller and larger roots of the equation
\begin{equation}
\triangle_{k}(q)=E_{k+1}(q)-E_{k}(q)=q^{k}(1-q)-\frac{1}{k(k+1)}=0,\,\,k\geq 3.
\end{equation}
\cite{S1978} shows that $\frac{k}{k+1}=\arg\max_{q}\triangle_{k}(q)$,\,\, $\triangle_{k}(\frac{k}{k+1})>0$,\,\, $\triangle_{k}(0)=\triangle_{k}(1)<0$
and that both $r_{1}(k)$ and $r_{2}(k)$ are increasing functions of $k$.
Combining points one to three, we can conclude that if $q<r_2(3)$ then $E_3<E_4$, if $q<r_2(4)$ then $E_4<E_5$, and so on. The inequality $r_{2}(3)<r_{2}(4)<\ldots$ completes the proof.
\end{proof}

\begin{comment}
\label{com:2}
It follows from the
proof of Proposition \ref{eq:prop2} that $E\left(k^{*}(p), p\right)$ is a non-decreasing function of $p,\,\, p\in (0,1)$.
\end{comment}

\section{{\bf Support of the Loss Function with Respect to $p$}}
\label{AA:0}
\begin{result}
\label{res:3}
$\displaystyle \sup_{p\in (0, 1-(1/3)^{1/3}]} L\left(k, p\right)=\sup_{p\in (0, 1)} L\left(k, p\right)$.
\end{result}
\begin{proof}
First we show that $\lim_{p\downarrow 0} E\left(k^{*}(p), p\right)=0$.  Recall, that according to \cite{S1978}, if $ 0<p\leq 1-(1/3)^{1/3}$
then $k^{*}$ is $1+[p^{-1/2}]$ or $2+[p^{-1/2}]$ and in this case from \eqref{eq:D} we have $\displaystyle E\left(k^{*}(p), p\right)=1-(1-p)^{k^{*}(p)}+\frac{1}{k^{*}(p)}$.
Specifically, for any $p$ satisfies $0<p\leq 1-(1/3)^{1/3}$, we have  $1>(1-p)^{k^{*}(p)}\geq(1-p)^{2+p^{-1/2}}$. From the facts that $\lim_{p\downarrow 0}(1-p)^2=1$ and
$\lim_{p\downarrow 0}\frac{1}{p^{1/2}}\log(1-p)=0$, it follow that $\lim_{p\downarrow 0}(1-p)^{k^{*}(p)}=1$.
Combining this with \cite{S1978} result that $k^{*}(p)\rightarrow \infty$ as $p\downarrow 0$, we prove that $\lim_{p\downarrow 0} E\left(k^{*}(p), p\right)=0$.
\\
{\it For $k=1$}. We have $L\left(1, p\right)=1-E\left(k^{*}(p), p\right)$. Hence, $\sup_{p\in (0, 1)} L\left(1, p\right)=1-\inf_{p\in (0, 1)}E\left(k^{*}(p), p\right)=
1-\lim_{p\downarrow 0}E\left(k^{*}(p), p\right)$, where the last equation follows from Comment \ref{com:2}. Therefore,
\begin{equation}
\label{eq:c1}
\displaystyle \sup_{p\in (0, 1-(1/3)^{1/3}]} L\left(1, p\right)=\sup_{p\in (0, 1)} L\left(1, p\right).
\end{equation}
{\it For $k\geq 2 $}.
Define $p_0=1-(1/3)^{1/3}$. For any $p\in (p_0,\,1)$, from \eqref{eq:D} and \eqref{eq:Loss} it follows that $L\left(k, p\right)=E\left(k, p\right)-1=1/k-(1-p)^k$ is an increasing function of $p$, and therefore from \eqref{eq:D} and \eqref{eq:E} it follows that
\begin{equation}
\label{eq:c2}
\sup_{p\in (p_0, 1)} L\left(k, p\right)=E\left(k, 1\right)-1=1/k.
\end{equation}
Again, from Comment \ref{com:2} and \eqref{eq:E}, it follows that $\lim_{p\downarrow 0} L\left(k, p\right)=1/k$.
Therefore,
\begin{equation}
\label{eq:c3}
\sup_{p\in (0, p_0]} L\left(k, p\right)\geq 1/k.
\end{equation}
Combining, \eqref{eq:c2} and \eqref{eq:c3}, we get
\begin{equation}
\label{eq:c5}
\sup_{p\in (0, p_0]} L\left(k, p\right)\geq \sup_{p\in (p_0, 1)} L\left(k, p\right).
\end{equation}
The equations \eqref{eq:c1} and \eqref{eq:c5} completes the proof.
\end{proof}

\section{{\bf Behavior of $p^{*}(k)$ and Unimodality of $L\left(k,\,p^{*}(k) \right)$ }}
\label{AAA:p}
\begin{result}
\label{res:2}
$\displaystyle p^{*}(k)=0$, for $k=1,2,\ldots,7$,\,\,\ $\displaystyle p^{*}(8)=1-\left(\frac{3}{8}\right)^{1/(8-3)}\approx 0.178$.
Moreover,
$L\left(k,\,p^{*}(k) \right)$ is a unimodal function of $k$ for $k\geq 1$.
\end{result}

\begin{proof} Define $\displaystyle q_0=1-p_0$.
Recall that it was shown in \eqref{eq:c3} that  $\displaystyle \sup_{p\in (0, p_0]} L\left(k, p\right)\geq 1/k$. Also we have $\displaystyle \lim_{p\downarrow 0} L\left(k, p\right)=1/k$.
Therefore, if $\displaystyle \sup_{p\in (0, p_0]} L\left(k, p\right)=1/k$, then $\displaystyle p^{*}(k)=0$, and if $\displaystyle \sup_{p\in (0, p_0]} L\left(k, p\right)>1/k$
then $\displaystyle p^{*}(k)>0$. Now $\displaystyle L\left(k, p\right)>1/k$ if and only if
\begin{equation}
\label{eq:ns}
q^{k^{*}}-q^k>\frac{1}{k^{*}}.
\end{equation}
It is clear, that the necessary condition for the inequality \eqref{eq:ns} is
\begin{equation}
\label{eq:k}
k^{*}<k.
\end{equation}
We know (\cite{S1978} and Proposition \ref{eq:prop2} above) that for $q>q_0$, $k^{*}$ is a piecewise non-increasing (non-decreasing ) function of $p$ ($q$) and $k^{*}$ is at least $3$.
From this and \eqref{eq:k}, it follows that $\displaystyle p^{*}(k)=0$, for $k=1,2,3$. Using Proposition \ref{eq:prop2}, we can verify that for $k=7$, there does not exist a $q$ ($q\in [q_0, 1)$) such that \eqref{eq:ns} holds.
From direct logic it follows that if there exists $q$ such that equation \eqref{eq:ns} holds for $k^{'}$, then \eqref{eq:ns} holds for any $k>k^{'}$. It is the same to say that if for any $q$ (in the appropriate range) the equation \eqref{eq:ns} does not hold for $k^{''}$, then it does not hold for any  $q$ and for any $k<k^{''}$.
Therefore, $\displaystyle p^{*}(k)=0$ and $\displaystyle L\left(k,\,p^{*}(k) \right)=1/k$,\, for $k=1,2,\ldots,7$.
\\
\noindent
Again using Proposition \ref{eq:prop2},
we can verify that for $k=8$, the equation \eqref{eq:ns} holds only for some values of $q$
that corresponds to $k^{*}=3$. Therefore, finding $\displaystyle \sup_{p\in (0, p_0]} L\left(8, p\right)$ is equivalent
to finding $\displaystyle \sup_{q^{*}_{2}<q<q^{*}_{3}}\left\{q^{3}-q^{8}\right\}$, where $\displaystyle q^{*}_{2}$ and
$\displaystyle q^{*}(8)=q^{*}_{3}$ are defined in Proposition \ref{eq:prop2}. A simple calculation then shows that
$\displaystyle q^{*}(8)=(3/8)^{1/5}$ (correspondingly $\displaystyle p^{*}(8)=1-(3/8)^{1/5}\approx 0.1781$ ) and
$\displaystyle L\left(8,\,p^{*}(8) \right)= 437/3152\approx 0.1386.$
\\
\noindent
From above, it follows that for $k\geq 8$, there exists a $q$ such that $q^{*}_{2}<q<q^{*}_{3}$ and equation \eqref{eq:ns} holds.
Now, for such a $q$, it follows from the last inequality from the last paragraph in the proof of Proposition \ref{eq:prop2} that
$\displaystyle L\left(9,\,p\right)-L\left(8,\,p\right)=E(9,\,p)-E(8,\,p)>0,\,\,p=1-q$.
 Therefore,
\begin{equation}
\label{eq:77}
L\left(8,\,p^{*}(8) \right)=\sup_{p\in (0,\,p_0]}L\left(8,\,p\right)<\sup_{p\in (0,\,p_0]}L\left(9,\,p\right)=L\left(9,\,p^{*}(9) \right).
\end{equation}
Proceeding by induction on $k=10,11,\ldots$ we complete the proof of unimodality of $L\left(k,\,p^{*}(k) \right)$ for $k\geq 1$.
\end{proof}

\begin{remark}
\label{Rem:1}
The function $\displaystyle p^{*}(k)$ is a decreasing function of $k$ for $k\geq 8$ and $\displaystyle p^{*}(k)$ has a form $1-(k^{*}/k)^{1/(k-k^{*})}$ for $k\geq 8$
and for $k^*$ that satisfied
equation \eqref{eq:ns} and the condition $q^{*}_{k^{*}-1}<q<q^{*}_{k^{*}}$ of Proposition \ref{eq:prop2}. But we have not proved it rigorously.
\end{remark}

\begin{lem}
\label{lem:m}
$k^{**}=8$.
\end{lem}
\begin{proof}
Follows immediately from Result \ref{res:2} and the proof that
$\displaystyle L\left(k, p^{*}(k)\right)$ is the unimodal function of $k$ with the minimum at $k=8$.
\end{proof}

\section{{\bf Matlab code}}
\label{A:a}

\begin{enumerate}
\item[(i)]
{\bf Matlab function ``MinMaxValue"}, which has input $K$ and uses all three functions $(ii), (iii), (iv)$ given below. \\
\#function minmax=MinMaxValue(K,U)\\
\#R=[];\\
\#step=1/1000000;\\
\#for k=1:1:K\\
  \#   output=zeros(1000001,3);\\
   \#  counter=0;\\
            \#for p=0:step:U\\
                \#counter=counter+1;\\
                \#prophet=ProphetNumberTestPerPerson(p); player=PlayerNumberTestPerPerson(p,k);\\
                \#l=player-prophet; output(counter,:)=[k p l];\\
            \#end\\

      \#m=max(output(:,3)); pl=find(output(:,3)==m); R=[R;output(pl,:)];\\
\#end\\
\#mm=min(R(:,3)); mmm=find(R(:,3)==mm); minmax=R(mmm,1);\\

\item[(ii)]
{\bf Matlab function ``OptimalGroupSize"}, which has input $p\in(0,1)$ and the output is the optimum group size $k^{*}(p)$ (based on \cite{S1978}):\\
\#function kOpt=OptimalGroupSize(p)\\
\#if $p<=1-(1/3)^{1/3}$;\\
   \#$q=1-p$; $w=p^{(-1/2)}$;\\
   \#int=floor(w); frl=w-int;\\
   \#$Ind=(frl<int/(2*int+frl))$;\\
   \#$k1=int+1$; k2=int+2;\\
   \#$f1=1/k1+1-q^{k1}$; $f2=1/k2+1-q^{k2}$;\\
   \#$f=min(f1,f2)$; $Ind1=(f==f1)$;\\
   \#$kOpt=((1+int)^{Ind}*(((k1)^{Ind1})*((k2)^{(1-Ind1)}))^{(1-Ind)}$;\\
\#else\\
   \#$kOpt=1$;\\
\#end

\item[(iii)]
{\bf Matlab function ``PlayerNumberTestPerPerson"} calculates the expected number of tests per person in a group of size $k$ with probability of infection $p$.\\
\#function player=PlayerNumberTestPerPerson(p,k)\\
\#q=1-p;\\
\#player$=1-q^k+1/k$;

\item[(iv)]
{\bf Matlab function ``ProphetNumberTestPerPerson"} calculates the expected number of tests per person in a group of the optimal size $k^{*}(p)$ with probability of infection $p$.\\
\#function prophet=ProphetNumberTestPerPerson(p)\\
\#kOpt=OptimalGroupSize(p);
\#q=1-p;
\#prophet$=1-q^kOpt+1/kOpt$;

\item[(v)]
{\bf Matlab function ``JeffreyOptGroupsize"} calculates the optimal group size under Jeffreys prior with upper bound for support of $p$ equal to $U$.\\
\#function kj=JeffreyOptGroupsize(U)\\
\#S=[];\\
\#for k=1:1:400\\
    \#s=PlayerNumTestPerBetta1(k,U,1/2,1/2); S=[S;k s];\\
\#end\\
\#m=min(S(:,2)); l=find(S(:,2)==m);kj=S(l,1);\\
\item[(vi)]
{\bf Matlab function ``PlayerNumTestPerBetta1"} calculates the expected number of tests under Beta prior with upper bound for support of $p$ equal to $U$.\\
\#function player=PlayerNumTestPerBetta1(k,U,a,b)\\
\# $c=@(p) p^{(a-1)}*(1-p)^{(b-1)}$;\\
\#$ f=@(p)(1-((1-p)^{k})+1./k)*p^{(a-1)}*(1-p)^{(b-1)}$;\\
\#$player=(quad(c,0,U))^{(-1)}*quad(f,0,U)$;
\end{enumerate}

\section{Relative efficiency}
\begin{table}[!h]
\caption{Relative efficiencies of restricted minimax and Bayesian designs} 
\begin{center}
 \begin{tabular}{ c|c cc|ccc|ccccc}
 \hline
\multicolumn{1}{c}{}&\multicolumn{3}{ c| }{$U=0.10$}&\multicolumn{3}{ c| }{$U=0.20$}&\multicolumn{3}{ c }{$U=0.30$} \\
\hline

 $p$                             & $0.01$ & $0.05$ & $0.10$ & $0.10$ & $0.15$ &$0.20$  & $0.20$  &$0.25$  & $0.30$  \\
   \hline
 $RE_{U}(p)$                      & $ 1.0342$ &$1.0830$ & $1.1694$ & $ 1.1694$ & $ 1.1853$ & $1.1655$ &  $ 1.1655$ & $ 1.1244$ & $1.0778$\\
 \hline
 $RE_{I_{U}}(p)$                      & $1.2732$ &$1$ & $1.0263$ & $1$ & $ 1.0122$ & $ 1.0232$ &  $ 1.0232$ & $1.0243$ & $1.0198$\\
 \hline
 $RE_{J_{U}}(p)$                      & $1.0778$ &$1.0429$ & $1.1190$ & $1.0263$ & $ 1.0516$ & $1.0621$ &  $1.0621$ & $1.0562$ & $1.0420$\\
   \hline
 $k^{*}(p)$                       &$11$&$5$&$4$&$4$&$3$&$3$&$3$&$3$&$3$\\
 \hline
 $k^{**}_{U}$                   &$8$&$8$&$8$&$8$&$8$&$8$&$8$&$8$&$8$\\
 \hline
 $k^{*}_{I_{U}}$                   &$5$&$5$&$5$&$4$&$4$&$4$&$4$&$4$&$4$\\
 \hline
 $k^{*}_{J_{U}}$                   &$7$&$7$&$7$&$5$&$5$&$5$&$5$&$5$&$5$\\
\end{tabular}
\label{table:5} 
\end{center}
\end{table}

{}


\begin{thebibliography}{}

\bibitem[\protect\citeauthoryear{Berger \it{et~al.}}{2009}]{B2009}
Berger, J. O., Bernardo, J. M. and Sun, D. (2009).
\newblock
The formal definition of reference priors.
\newblock {\it Annals of Statistics } {\textbf 37,} 905--938.


\bibitem[\protect\citeauthoryear{Berger \it{et~al.}}{2012}]{B2012}
Berger, J. O., Bernardo, J. M. and Sun, D. (2012).
\newblock
Objective priors for discrete parameter spaces.
\newblock {\it J. Amer. Statist. Assoc. } {\textbf 107,} 636--648.

\bibitem[\protect\citeauthoryear{Bernardo}{1979}]{B1979}
Bernardo, J. M. (1979).
\newblock
Reference posterior distributions for Bayesian inference.
\newblock {\it J. Royal Statistical Society B } {\textbf 41,} 113--147(with discussion).

\bibitem[\protect\citeauthoryear{Clarke and Barron}{1990}]{CB1990}
Clarke, B. S., Barron, A. R. (1990).
\newblock
Information-theoretic asymptotics of Bayes methods.
\newblock {\it IEEE Trans. Inform. Theory } {\textbf 36,} 453--471.



\bibitem[\protect\citeauthoryear{De Bonis \it{et~al.}}{2005}]{D2005}
De Bonis, A., Gasieniec, L., Vaccaro, U. (2005).
\newblock
Optimal Two-Stage Algorithms for Group Testing Problems.
\newblock {\it SIAM J. Comput. } {\textbf 34,} 1253--1270.




\bibitem[\protect\citeauthoryear{Delaigle and Hall}{2012}]{DH2012}
Delaigle, A., Hall, P. (2012).
\newblock Nonparametric regression with homogeneous group testing data.
\newblock {\it Ann. Statist.} {\textbf 40,} 131--158.






\bibitem[\protect\citeauthoryear{Dorfman}{1943}]{Dorfman1943}
Dorfman, R. (1943).
\newblock The detection of defective members of large populations.
\newblock {\it The Annals of Mathematical Statistics } {\textbf 14,} 436--440.

\bibitem[\protect\citeauthoryear{Du and Hwang}{1999}]{Dh1999}
Du, D., Hwang, F. K. (1999).
\newblock Combinatorial Group Testing and its Applications. {\it World
Scientific, Singapore}.

\bibitem[\protect\citeauthoryear{Du and Hwang}{2006}]{Dh2006}
Du, D., Hwang, F. K. (2006).
\newblock Pooling Design and Nonadaptive Group Testing: Important Tools for DNA Sequencing. {\it World
Scientific, Singapore}.

\bibitem[\protect\citeauthoryear{Feller}{1950}]{F1950} Feller, W. (1950).
\newblock
{An introduction to probability theory and its application}.
{\it New York: John Wiley \& Sons}.

\bibitem[\protect\citeauthoryear{Feller}{1968}]{F1968} Feller, W. (1968).
\newblock
{An introduction to probability theory and its application}. Third Eition.
{\it New York: John Wiley \& Sons.}

\bibitem[\protect\citeauthoryear{Ferguson}{1967}]{F1967} Ferguson, T. S. (1967).
\newblock
{Mathematical Statistics : A Decision Theoretic Approach}.
{\it New York and London: Academic Press.}


\bibitem[\protect\citeauthoryear{Finucan}{1964}]{F1964}
Finucan, H. M. (1964).
\newblock
The blood testing problem.
\newblock
{\it Applied Statistics} {\textbf 13,} 43--50.

\bibitem[\protect\citeauthoryear{Gastwirth and Johnson}{1994}]{G1994}
Gastwirth, J. and Johnson, W. (1994).
\newblock
Screening with cost effective
quality control: Potential applications to HIV and drug testing.
\newblock
{\it J. Amer. Statist. Assoc.} {\textbf 89,} 972--981.

\bibitem[\protect\citeauthoryear{Golan \it{et~al.}}{2012}]{G2012}
Golan, D., Erlich, Y., Rosset, S. (2012).
\newblock
Weighted Pooling - Practical and Cost Effective Techniques for Pooled High Throughput Sequencing.
\newblock
{\it Bioinformatics} {\textbf 28,} i197--i206.


\bibitem[\protect\citeauthoryear{Hughes-Oliver}{2006}]{H2006}
Hughes-Oliver, J.M. (2006).
\newblock
Pooling experiments for blood screening and drug discovery.
\newblock
In
Screening: Methods for Experimentation in Industry, Drug Discovery, and Genetics, edited
by Dean, A.M. and Lewis, S.M. and published by {\it Springer-Verlag New York, Inc}.


\bibitem[\protect\citeauthoryear{Jeffreys}{1946}]{J1946}
Jeffreys, H. (1946).
\newblock
An Invariant Form for the Prior Probability in Estimation Problems.
\newblock
{\it Proceedings of the Royal Society of London, Series A} {\textbf 186,} 453--461.

\bibitem[\protect\citeauthoryear{Lehmann and Casella}{1998}]{LC1998}
Lehmann, E. L., Casella, G. (1998).
\newblock
Theory of Point Estimation, 2nd ed.
\newblock
{\it Springer, New York}.









\bibitem[\protect\citeauthoryear{Litvak \it{et~al.}}{1994}]{L1994}
Litvak, E., Tu, X., and Pagano, M. (1994).
\newblock
Screening for the presence
of a disease by pooling sera samples.
\newblock
{\it J. Amer. Statist. Assoc.} {\textbf 89,} 424--434.

\bibitem[\protect\citeauthoryear{McMahan \it{et~al.}}{2012}]{M2012}
McMahan, C., Tebbs, J., and Bilder, C. (2012).
\newblock
Informative Dorfman screening.
\newblock
{\it Biometrics} {\textbf 68,} 287--296.


\bibitem[\protect\citeauthoryear{Moore \it{et~al.}}{2000}]{M2000}
Morre, S. A., Meijer, C. J. L. M., Munk, C., Kr$\ddot{u}$ger-Kjaer, S., Winther, J. F., Jorgensens, H. O., Van den Brule, A. J. C. (2000).
\newblock
Pooling of Urine Specimens for Detection of Asymptomatic Chlamydia trachomatis Infections by PCR in a Low-Prevalence Population: Cost-Saving Strategy for Epidemiological Studies and Screening Programs.
\newblock {\it J. Clin. Microbiol.} \textbf{38}, 1679–1680.






\bibitem[\protect\citeauthoryear{Pilcher \it{et~al.}}{2004}]{P2004}
Pilcher, C. D., Price, M. A., Hoffman, I. F., Galvin, S., Martinson, F. E.,
Kazembe, P. N., Eron, J. J., Miller, W. C., Fiscus, S. A., Cohen, M. S. (2004).
\newblock
Frequent detection of acute primary HIV infection in men in Malawi.
\newblock
{\it AIDS} {\textbf 18,} 517--524.

\bibitem[\protect\citeauthoryear{Robbins}{1952}]{R1952}
Robbins, H. (1952).
\newblock
Some aspects of the sequential design of experiments.
\newblock
{\it Bulletin of the American Mathematics Society} {\textbf 58,} 527--535.

\bibitem[\protect\citeauthoryear{Samuels}{1978}]{S1978}
Samuels, S. M. (1978).
\newblock The exact solution to the two-stage group-testing problem.
\newblock {\it Technometrics} {\textbf 20,} 497--500.

\bibitem[\protect\citeauthoryear{Schneider and Tang}{1990}]{ST1990}
Schneider, H., Tang, K. (1990).
\newblock Adaptive procedures fot the two-stage group-testing problem based on prior distributions and costs.
\newblock {\it Technometrics } {\textbf 32,} 397--405.

\bibitem[\protect\citeauthoryear{Shemyakin}{2014}]{S2014}
Shemyakin, A. (2014).
\newblock Hellinger distance and non-informative priors.
\newblock {\it Bayesian Analysis }. To appear.



\bibitem[\protect\citeauthoryear{Sobel and Groll}{1959}]{SG1959}
Sobel, M., Groll, P. A. (1959).
\newblock Group testing to eliminate efficiently all defectives in a binomial sample.
\newblock {\it Bell System Tech. J.} {\textbf 38,} 1179--1252.

\bibitem[\protect\citeauthoryear{Sobel and Groll}{1966}]{SG1966}
Sobel, M., Groll, P. A. (1966).
\newblock  Binomial group-testing with an unknown proportion of defectives.
\newblock {\it Technometrics } {\textbf 8,} 631--656.


\bibitem[\protect\citeauthoryear{Tamashiro \it{et~al.}}{1993}]{T1993}
Tamashiro, H., Maskill, W., Emmanuel, J., Fauquex, A., Sato, P., Heymann, D.  (1993).
\newblock
Reducing the cost of HIV antibody testing.
\newblock
{\it Lancet} {\emph 342,} 87--90.




\bibitem[\protect\citeauthoryear{Tebbs \it{et~al.}}{2013}]{T2013}
Tebbs, J., McMahan, C., and Bilder, C.  (2013).
\newblock
Two-stage hierarchical group testing for
multiple infections with application to the Infertility Prevention Project.
\newblock
{\it Biometrics} {\bf 69}, 1064--1073.

\bibitem[\protect\citeauthoryear{Ungar}{1960}]{U1960}
Ungar, P. (1960). Cutoff points in group testing.
{\it Comm. Pure Appl. Math.} {\bf 13,} 49--54.

\bibitem[\protect\citeauthoryear{van den Berg \it{et~al.}}{2013}]{v2013}
van den Berg, E., Cand$\grave{e}s$, E., Chinn, G., Levin, C., Olcott, P. D., Sing-Long, C. (2013).
\newblock
Single-photon sampling architecture for solid-state imaging sensors
\newblock {\it Proceedings of the National Academy of Sciences USA} \textbf{110(30)}, E2752--E2761.


\bibitem[\protect\citeauthoryear{Westreich \it{et~al.}}{2008}]{W2008}
Westreich,  D. J., Hudgens, M. G., Fiscus, S. A., Pilcher, C. D. (2008).
\newblock
Optimizing screening for acute human immunodeficiency virus infection with pooled nucleic acid amplification tests.
\newblock {\it J. Clin. Microbiol.} \textbf{46}, 1785–-1792.

\end{thebibliography}
\end{document}